\documentclass{amsart}[12pt]
\usepackage{amsmath, amsthm, amscd, amssymb, amsfonts}
\UseRawInputEncoding
\usepackage{graphics,epsfig}
\usepackage{csquotes}

\makeatletter \oddsidemargin.9375in \evensidemargin \oddsidemargin
\newtheorem{theorem}{Theorem}[section]
\newtheorem{lem}[theorem]{Lemma}

\newtheorem{cor}[theorem]{Corollary}
\theoremstyle{definition}
\newtheorem{defn}[theorem]{Definition}
\newtheorem{ex}[theorem]{Example}

\theoremstyle{remark}
\newtheorem{rem}[theorem]{Remark}
\numberwithin{equation}{section}
\begin{document}
\title[Bernoulli polynomials series via analytic summability of functions]
{Bernoulli polynomials series via analytic summability of functions}

\author[M.H. Hooshmand]{M.H. Hooshmand }
\author[Soodeh Mehboodi]{Soodeh Mehboodi}

\address{Department of Mathematics, Shi.C., Islamic Azad University, Shiraz, Iran.}
\email{MH.Hooshmand@iau.ac.ir}

\address{Zand Institute of Higher Education, Shiraz, Iran.}
\email{\tt soodehmehboodi@yahoo.com}


\date{}
----------------------------------------------
\begin{abstract}
 Analytic summability of real and complex functions was introduced in 2016. In the topic, the Bernoulli numbers and polynomials were used for defining analytic summand of a given function with a power series on an open domain $D$.
 In this paper, we study derivatives and integrals of the analytic summand functions and show that the derivatives are the same Bernoulli polynomials series $\sum_{n=0}^{\infty}c_nB_n(z)$ up to a unit forward shift. Therefore, the topic of the analytic summability is a platform for studying Bernoulli polynomials series. In the way, we obtain many new results for the Bernoulli polynomials series such as some related series convergent tests and upper bounds for the Bernoulli polynomials and the mentioned series. For instance, we observe that $\sum_{n=0}^{\infty}c_nB_n(z)$ is absolutely convergent on $\mathbb{C}$, if the numerical series $\sum_{n=0}^{\infty}\frac{n!}{\pi^n}c_n$ is absolutely convergent. Also, we present some applications and various examples of the topic such as the inequality
 $$|\sum_{n=1}^{\infty}\frac{\pi^n}{n!n^p}B_n(z)|\leq 2e^{\pi|z|}\zeta(p)$$
held for every fixed real number $p>1$ and all  $z\in \mathbb{C}$.
\end{abstract}

\maketitle
\section{introduction and preliminaries \noindent}
Bernoulli  polynomials $B_n(z)$ and numbers have many properties and applications in various branches of mathematics and other sciences. In the way, the first and second Bernoulli numbers (which are rational) play important roles and some bounds for them were presented (e.g., see \cite{alzer, hua feng, sandor}). But, to the best of our knowledge, there are no serious studies about the Bernoulli polynomials series $\sum_{n=0}^{\infty}c_nB_n(z)$ yet.
\par
The first author introduced {\it limit summability} and related {\it limit summand} functions  in 2001. Also, in order to study the derivative of the limit summand functions he introduced the functional sequence $f_{\sigma^\prime_n}(x)$, that is related to the Euler-type constants (see \cite{Euler type, limit, sandor2}). Moreover, he and his colleague proved some criteria for convergence of $f_{\sigma^\prime_n}(x)$ which give some sufficient conditions for existence of such generalized Euler constants. \par
On the other hand, some of the well known functions such as exponential, hyperbolic and trigonometric functions are not limit summable. To this purpose, Hooshmand introduced {\it analytic summability} of functions \cite{analytic} to remove this problem. He applied the Bernoulli polynomials and numbers to define the summability. Moreover, in 2017, he introduced functional sequential and trigonometric summability of functions which is a generalization of analytic summability. Recently, in \cite{soodeh} he and his colleagues improved some upper bounds and inequalities for analytic summand functions.\par
In this paper, we study derivatives and integrals of analytic summand functions and their related upper bounds. Moreover, we arrive at some interesting results and connections such as relations to the Bernoulli polynomials series and we show that the derivative of analytic summand functions is the same Bernoulli polynomials series $\sum_{n=0}^{\infty}c_nB_n(z)$ up to a unit forward shift.  So it motivated us to study such series by applying analytic summability.  \par
We recall some definitions and properties of Bernoulli polynomials and
numbers of \cite{handbook, apostel, polynomials} below. The  Bernoulli polynomials $B_n(z)$  ($n=0,1,2,...; z\in\mathbb{C}$) are defined by
\begin{align}\label{bnzformula}
\frac{te^{zt}}{e^t-1}=\sum_{n=0}^{\infty}B_n(z)\frac{t^n}{n!} \quad (|t|<2\pi).
\end{align}
 The first few polynomials and numbers are
\begin{align*}
&B_0(z)=1, \quad B_1(z)=z-\frac{1}{2}, \quad B_2(z)=z^2-z+\frac{1}{6}, \quad B_3(z)=z^3-\frac{3}{2}z^2+\frac{1}{2}z.
\end{align*}
The first (resp. second) Bernoulli numbers are defined by $B_n:=B_n^-=B_n(0)$ (resp. $\mathfrak{b}_n:=B_n^+=B_n(1)$). All Bernoulli numbers are rational numbers.
Explicit formula for $B_n(z)$ is given by
\begin{align}\label{bnenzformula}
B_n(z)=\sum_{k=0}^n\binom{n}{k}B_kz^{n-k}.
\end{align}
Also, Bernoulli  polynomials follow many relations such as the difference functional equation below
\begin{align}
B_n(z+1)-B_n(z)=nz^{n-1}.
\end{align}
 Here, we recall  notations, a basic definition and an important criteria for analytic summability from \cite{analytic}.
\begin{align}\label{sigmaoprator}
\sigma_\mathcal{A}(z^n)=\sigma(z^n)=S_n(z):=\frac{B_{n+1}(z+1)-\mathfrak{b}_{n+1}}{n+1}, \quad z\in\mathbb{C}, n\geq0,
\end{align}
\begin{align}
\beta_{nk}=\beta_{n,k}:=\binom{n+1}{k}\frac{\mathfrak{b}_{n+1-k}}{n+1}=\frac{n!}{k!(n+1-k)!}\mathfrak{b}_{n+1-k},
\end{align}
\begin{align}\label{snz formula01}
S_n(z)=\sum_{k=1}^{n+1}\beta_{nk}z^k, \quad z\in\mathbb{C}, n\geq0.
\end{align}
{\bf Note.} It is worth noting that $\sigma(z^n)$ is just a formal symbol and we can not put a special $z$ in it directly, so for this purpose we should apply $S_n(z)$ (see \cite{analytic}).
\begin{defn}\label{defn01}
Let $f(z)=\sum_{n=0}^{\infty}c_nz^n$ be a complex or real analytic function defined on an open domain $D$ (indeed, $D$ is an open disk with the center of zero). We call $f$  \textquote{analytically summable} at $z_0$ (resp. absolutely analytically summable) if the series
 \begin{align}\label{fsigma def}
 f_{\sigma_\mathcal{A}}(z_0)=f_\sigma(z_0)=\sum_{n=0}^{\infty}c_n\sigma(z^n) = \sum_{n=0}^{\infty}c_nS_n(z),
 \end{align}
  is convergent (resp. absolutely convergent). We call $f$  analytically summable on $E\subseteq D$ if it is analytically summable at every point of $E$. The function $f_{\sigma_\mathcal{A}}=f_\sigma$ (with the largest possible domain) is called analytic summand (function) of $f$. If $f$ is analytically summable on the whole $\mathbb{C}$, then we call $f$  \textquote{entirely summable}.
\end{defn}
If $f$ is analytically summable on $D$ then $f_\sigma$ satisfies the difference functional equation below (e.g., see \cite{bruce})
\begin{align}\label{function equation}
f_\sigma(z)=f(z)+f_\sigma(z-1) \quad;\quad z\in D\cap D+1.
\end{align}
In \cite{analytic, soodeh}, Hooshmand and his students proved some upper bounds and criteria for analytic summand functions by applying the following bounds from \cite{alzer, hua feng, sandor}.
\begin{align}\label{bnbound1}
\frac{2(2n)!}{(2\pi)^{2n}}\frac{1}{1-2^{-2n}}<|B_{2n}|<\frac{2(2n)!}{(2\pi)^{2n}}\frac{1}{1-2^{1-2n}},
\end{align}
\begin{align}\label{bnbound2}
\frac{2(2n)!}{(2\pi)^{2n}}\frac{1}{1-2^{\alpha-2n}}<|B_{2n}|<\frac{2(2n)!}{(2\pi)^{2n}}\frac{1}{1-2^{\beta-2n}},\;\;
\end{align}
where $\beta=2+\frac{\ln (1-\frac{6}{\pi^2})}{\ln (2)}\approx 0.6491...$ and $\alpha$ is another constant.
\begin{align}\label{bnbound3}
\frac{2(2n)!}{\pi^{2n}(2^{2n}-1)} < |B_{2n}| < \frac{2(2^{2j}-1)}{2^{2j}} \zeta(2j) \frac{2(2n)!}{\pi^{2n}(2^{2n}-1)},
\end{align}
where $\zeta$ is the zeta Rimman function, $j$ is a fixed positive integer and $n\geq j$.
\par Here we state a summary of them as a theorem. Assume that $f(z)=\sum_{n=0}^{\infty}c_nz^n$ then put
\begin{align*}
&Abs(f(z))=\sum_{n=0}^{\infty}|c_n||z|^n,\quad Abs_{\frac{!}{\pi}}(f)=\sum_{n=0}^{\infty}\frac{n!}{\pi^n}|c_n|\\&
Abs^e_{\frac{!}{\pi}}(f)=\sum_{\substack {n=0\\n\; \text{is even}}}^{\infty}\frac{n!}{\pi^n}|c_n|,\quad Abs^o_{\frac{!}{\pi}}(f)=\sum_{\substack {n=0\\n\; \text{is odd}}}^{\infty}\frac{n!}{\pi^n}|c_n|.
\end{align*}
\begin{theorem}\label{review thm}
Let $f(z) = \sum_{n=0}^{\infty}c_n z^n$ be an analytic function, if the series $\sum_{n=0}^{\infty}\frac{n!}{\pi^n}c_n$ is absolutely convergent, then $f$ is  absolutely entirely summable. Moreover, the analytic summand function $f_\sigma$  is entire and it has the following representation
\begin{align}\label{fsigma closed form}
f_\sigma(z)=\sum_{n=1}^{\infty}\sigma_nz^n=\sum_{n=1}^{\infty}\frac{1}{n!}\left(\sum_{j=0}^{\infty}\frac{(j+n-1)!}{j!}\mathfrak{b}_jc_{j+n-1}\right)z^n, \quad z\in \mathbb{C},
\end{align}
where $\sigma_n:=\displaystyle\lim_{N\to\infty}\sigma_{n,N} =\sum_{k=n-1}^{\infty}\beta_{k,n}c_k$.\\
In addition, $f_\sigma$ satisfies the following inequalities
\begin{align}\label{new fsigma bounds}
|f_\sigma(z)| &\leq \frac{1}{2} Abs(f(z)) + \frac{1}{3} Abs(F(z)) + \frac{2}{3\pi} \sinh (\pi|z|) Abs^e_{\frac{!}{\pi}}(f) \notag\\&+  \frac{2}{3\pi} (\cosh(\pi|z|)-1) Abs^o_{\frac{!}{\pi}}(f) \\&
\leq \frac{1}{2} Abs(f(z)) + \frac{12-\pi^2}{12} Abs(F(z)) + \frac{\pi}{12} \sinh (\pi|z|) Abs^e_{\frac{!}{\pi}}(f) \notag\\&+   \frac{\pi}{12} (\cosh(\pi|z|)-1) Abs^o_{\frac{!}{\pi}}(f) \notag\\&
\leq \frac{2}{\pi} (e^{\pi|z|}-1) Abs_{\frac{!}{\pi}}(f);\quad z\in\mathbb{C}.\notag
\end{align}
where $F(z) = \sum_{n=0}^{\infty}\frac{c_n}{n+1}z^{n+1}$.
\end{theorem}
\begin{proof}
See Theorem 4.1 of \cite{analytic} and Corollary 2.2 of \cite{soodeh}.
\end{proof}
It is worth noting that the above inequalities have several results and special inequalities such as the following (see Example 3.2 of \cite{soodeh})
\begin{align*}
1^p + 2^p + 3^p + ... + r^p &\leq \frac{1}{2}r^p + \frac{1}{3}\frac{r^{p+1}}{p+1} + \frac{2}{3}\frac{p!}{\pi^{p+1}}\sinh(\pi r),\;p\in\mathbb{N}_{e},
\end{align*}
and analogously for the odd case.

\section{derivative and integral of analytic summand functions}
In this section we study the derivative and integral of analytic summand functions.

\begin{theorem}\label{thm derivative}
Let $c_n$ be a complex sequence and put $f(z)=\sum_{n=0}^{\infty}c_nz^n$. If the series $\sum_{n=0}^{\infty}\frac{n!}{
\pi^n}c_n$ is absolutely convergent, then $f$ is absolutely entirely summable and:

(a) We have the first derivative formula
\begin{align}\label{derivative formula}
(f_\sigma)'(z) = \sum_{n=1}^{\infty}nc_nS_{n-1}(z) + \lambda(f),\; z\in\mathbb{C},
\end{align}
where $\lambda(f)=\sum_{n=0}^{\infty}c_n\mathfrak{b}_n$.

More generally, for every integer $m\ge 0$,
\begin{align}\label{mderivative formula}
(f_\sigma)^{(m)}(z)
&= \sum_{n=m}^{\infty} \frac{n!}{(n-m)!}\, c_n \, S_{n-m}(z) + \lambda_m(f), \qquad z\in\mathbb{C},
\end{align}
where the generalized constant is defined as
\begin{align}\label{lambdam def}
\lambda_0(f) :=0\; , \; \lambda_m(f) := \sum_{n=m-1}^{\infty} \frac{n!}{(n-m+1)!}\, \mathfrak{b}_{n-m+1}\, c_n\; (m\ge 1) .
\end{align}
Equivalently, the two terms in \eqref{mderivative formula} merge into the elegant Bernoulli-polynomial form
\begin{align}\label{mderivative merged}
(f_\sigma)^{(m)}(z) = \sum_{n=m-1}^{\infty} \frac{n!}{(n-m+1)!}\, c_n \, B_{n-m+1}(z+1), \qquad z\in\mathbb{C},\; m\ge 1.
\end{align}
(For $m=1$, \eqref{lambdam def} gives $\lambda_1(f)=\lambda(f)$, so \eqref{mderivative formula} reduces to \eqref{derivative formula}.)

(b) The integral formula holds:
\begin{align}\label{integral formula}
\int f_\sigma(z) dz = \sum_{n=0}^{\infty} \frac{c_n}{n+1}S_{n+1}(z)-\mu(f)z +C,\; z\in\mathbb{C},
\end{align}
where $\mu(f)=\sum_{n=1}^{\infty}\mathfrak{b}_n\frac{c_{n-1}}{n}$ and $C$ is an arbitrary constant.
\end{theorem}

\begin{proof}
(a) Firstly, Theorem \ref{review thm} implies that $f_\sigma$ is an entire function. Then we calculate the partial sum of the functional series $\sum_{n=1}^{\infty}nc_nS_{n-1}(z)$ as follows
\begin{align}
\sum_{n=1}^{N}nc_nS_{n-1}(z) &= \sum_{n=1}^{N}nc_n\sum_{k=1}^{n}\beta_{n-1,k}z^k\\&
= \sum_{n=1}^{N}\sum_{k=1}^{n}nc_n\beta_{n-1,k}z^k
= \sum_{n=2}^{N+1}\sum_{k=n-2}^{N-1}(k+1)c_{k+1}\beta_{k,n-1}z^{n-1}\notag\\&
= \sum_{n=2}^{N+1}\sum_{k=n-2}^{N-1}\frac{(k+1)!}{(n-2)!(k+2-n)!}\mathfrak{b}_{k+2-n}c_{k+1}z^{n-1}\notag\\&
=  \sum_{n=2}^{N+1}\frac{1}{(n-1)!}\sum_{j=0}^{N-1}\frac{(j+n-1)!}{j!}\mathfrak{b}_jc_{j+n-1}z^{n-1}.\notag
\end{align}
Now letting $N\to\infty$ and doing attention to (\ref{fsigma closed form}) imply that
\begin{align}\label{fsigmaprimeN}
\sum_{n=1}^{\infty}nc_nS_{n-1}(z) &= \sum_{n=2}^{\infty}\frac{1}{(n-1)!}\sum_{j=0}^{\infty}\frac{(j+n-1)!}{j!}\mathfrak{b}_jc_{j+n-1}z^{n-1}\notag\\&
\sum_{n=1}^{\infty}\frac{1}{(n-1)!}\sum_{j=0}^{\infty}\frac{(j+n-1)!}{j!}\mathfrak{b}_jc_{j+n-1}z^{n-1}-\sum_{j=0}^{\infty}\mathfrak{b}_jc_j\notag\\&
=(f_\sigma)^\prime(z)-\sum_{j=0}^{\infty}\mathfrak{b}_jc_j.
\end{align}

We now establish the higher derivative formula by comparing the coefficients of the power-series expansions of both sides of \eqref{mderivative merged}.
By Theorem \ref{review thm},
\[
f_\sigma(z)=\sum_{n=1}^{\infty}\sigma_nz^n,
\]
where
\[
\sigma_n
=\sum_{k=n-1}^{\infty}\beta_{k,n}c_k
=\frac{1}{n!}\sum_{j=0}^{\infty}
\frac{(j+n-1)!}{j!}\mathfrak{b}_jc_{j+n-1}.
\]

Fix an integer $m\ge1$, and define
\[
g_m(z):=
\sum_{n=m-1}^{\infty}
\frac{n!}{(n-m+1)!}\,
c_n\,B_{n-m+1}(z+1).
\]

The coefficient of $z^k$ in $(f_\sigma)^{(m)}(z)$ is
\begin{align}
[z^k](f_\sigma)^{(m)}
&=
\frac{(k+m)!}{k!}\sigma_{k+m}\notag\\
&=
\frac{1}{k!}
\sum_{j=0}^{\infty}
\frac{(j+k+m-1)!}{j!}
\mathfrak{b}_jc_{j+k+m-1}.
\label{coeff derivative}
\end{align}

On the other hand, since
\[
B_r(z+1)
=
\sum_{\ell=0}^{r}
\binom{r}{\ell}
\mathfrak{b}_\ell
z^{\,r-\ell},
\]
the coefficient of $z^k$ in $g_m(z)$ is
\begin{align*}
[z^k]g_m
&=
\sum_{n=m-1}^{\infty}
\frac{n!}{(n-m+1)!}\,
c_n
\binom{n-m+1}{\,n-m+1-k\,}
\mathfrak{b}_{\,n-m+1-k}\\
&=
\sum_{n=m+k-1}^{\infty}
\frac{n!}{k!(n-m+1-k)!}
\mathfrak{b}_{\,n-m+1-k}c_n.
\end{align*}
Putting $j=n-m+1-k$ (equivalently, $n=j+k+m-1$), we obtain
\begin{align}
[z^k]g_m
&=
\frac{1}{k!}
\sum_{j=0}^{\infty}
\frac{(j+k+m-1)!}{j!}
\mathfrak{b}_jc_{j+k+m-1},
\label{coeff Gm}
\end{align}
which is identical to \eqref{coeff derivative}. Therefore the Taylor coefficients of the entire functions $(f_\sigma)^{(m)}$ and $g_m$ coincide. By the uniqueness theorem for power series,
\[
(f_\sigma)^{(m)}(z)
=
g_m(z)
=
\sum_{n=m-1}^{\infty}
\frac{n!}{(n-m+1)!}
c_n
B_{n-m+1}(z+1).
\]

Finally, using
\[
B_r(z+1)
=
rS_{r-1}(z)+\mathfrak{b}_r,
\qquad r\ge0,
\]
(where $S_{-1}(z)\equiv0$), we obtain
\begin{align*}
(f_\sigma)^{(m)}(z)
&=
\sum_{n=m-1}^{\infty}
\frac{n!}{(n-m+1)!}
c_n
\Bigl((n-m+1)S_{n-m}(z)+\mathfrak{b}_{\,n-m+1}\Bigr)\\
&=
\sum_{n=m}^{\infty}
\frac{n!}{(n-m)!}
c_nS_{n-m}(z)
+
\sum_{n=m-1}^{\infty}
\frac{n!}{(n-m+1)!}
\mathfrak{b}_{\,n-m+1}c_n.
\end{align*}
The second sum is precisely $\lambda_m(f)$, completing the proof of
\eqref{mderivative formula} and \eqref{mderivative merged}.

(b) A simple calculation shows that the partial sum of the functional series $\sum_{n=0}^{N}\frac{c_n}{n+1} S_{n+1}(z)$ is equal to
\begin{align*}
\sum_{n=0}^{N}\frac{c_n}{n+1} S_{n+1}(z) & = \sum_{n=0}^{N}\frac{c_n}{n+1}\sum_{k=1}^{n+2}\beta_{n+1,k}z^k
= \sum_{n=0}^{N}\sum_{k=1}^{n+2}\frac{c_n}{n+1}\beta_{n+1,k}z^k\\&
=\sum_{k=1}^{N+1}\frac{c_{k-1}}{k}\beta_{k,1}z + \sum_{n=1}^{N+1}\sum_{k=n}^{N+1}\frac{c_{k-1}}{k}\beta_{k,n+1}z^{n+1}\\&
= \sum_{k=1}^{N+1}\frac{c_{k-1}}{k}\mathfrak{b}_k z +\sum_{n=1}^{N+1}\sum_{k=n}^{N+1}\frac{(k-1)!}{(n+1)!(k-n)!}\mathfrak{b}_{k-n}c_{k-1}z^{n+1}\\&
= \sum_{k=1}^{N+1}\frac{c_{k-1}}{k}\mathfrak{b}_k z +\sum_{n=1}^{N+1}\frac{1}{(n+1)!}\sum_{j=0}^{N+1}\frac{(j+n-1)!}{j!}\mathfrak{b}_jc_{n+j-1}z^{n+1}.
\end{align*}
Letting $N\to\infty$ and using (\ref{fsigma closed form}) gives
\begin{align*}
\sum_{n=0}^{\infty}\frac{c_n}{n+1} S_{n+1}(z)  &= \sum_{k=1}^{\infty}\frac{c_{k-1}}{k}\mathfrak{b}_k z +\sum_{n=1}^{\infty}\frac{1}{(n+1)!}\sum_{j=0}^{\infty}\frac{(j+n-1)!}{j!}\mathfrak{b}_jc_{n+j-1}z^{n+1}\\&
=\sum_{k=1}^{\infty}\frac{c_{k-1}}{k}\mathfrak{b}_k z  + \int f_\sigma(z) dz+C.
\end{align*}
Thus $\int f_\sigma(z)dz = \sum_{n=0}^{\infty}\frac{c_n}{n+1}S_{n+1}(z) - \sum_{k=1}^{\infty}\frac{\mathfrak b_k}{k}c_{k-1}z + C$, which, upon defining $\mu(f)=\sum_{n=1}^{\infty}\frac{\mathfrak b_n}{n}c_{n-1}$, yields \eqref{integral formula}.
\end{proof}

\begin{rem}
The above theorem shows the interesting property of the derivative and integral of the functional sequence $\sum_{n=0}^{\infty}c_nS_n(z)$ which is simlar to power series (see \cite{handbook2}). Notice that, in the general case, we do not know that the series is a power series or not, but here the answer is positive (because $\sum_{n=0}^{\infty}\frac{n!}{\pi^n}c_n$ is absolutely convergent).\\
Also, notice that by using Theorem \ref{thm derivative} we have
 \begin{align*}
 \mu(f)=\sum_{n=1}^{\infty}\mathfrak{b}_n\frac{c_{n-1}}{n}=\sum_{n=0}^{\infty}\frac{\mathfrak{b}_{n+1}}{n+1}c_n.
 \end{align*}
 On the other hand, by applying (\ref{integral formula}) we obtain
 \begin{align*}
 \int_{0}^{1}f_\sigma(z)dz=\sum_{n=0}^{\infty}\frac{c_{n}}{n+1}+\mu(f)=\sum_{n=0}^{\infty}\frac{\mathfrak{b}_{n+1}+1}{n+1}c_n.
 \end{align*}
 So, we have
 \begin{align}
 \mu(f)&= \int_{0}^{1}f_\sigma(z)dz-\sum_{n=0}^{\infty}\frac{c_{n}}{n+1}\\&
 =\psi(1)-\psi(0)-F(1)\notag\\&
 =(\psi-F)(1)-(\psi-F)(0),\notag
 \end{align}
 where $\psi$ is an arbitrary primitive function of $f_\sigma(z)$ and $F(z)=\sum_{n=0}^{\infty}\frac{c_n}{n+1}z^{n+1}$.
\end{rem}

\begin{ex}
The function
\[
f(z)=\left(\frac{\pi z}{2}+\frac{\pi^2z^2}{4}\right)e^{\pi z/2}=\sum_{n=0}^{\infty}\frac{\pi^n n^2}{n!2^{n}}z^n
\]
is absolutely entirely summable (see Theorem \ref{review thm}). From Theorem \ref{review thm}, its analytic summand has the explicit power series
\[
f_\sigma(z) = \frac{1}{\pi}\sum_{n=1}^{\infty}\frac{\pi^{n}}{n!}
\sum_{j=0}^{\infty}\frac{\pi^j(j+n-1)^2}{2^{j+n-1}j!}\mathfrak{b}_j \, z^n,\qquad z\in\mathbb{C}.
\tag{1}
\]
Since the power series converges absolutely, termwise differentiation and integration are justified.

Differentiating (1) \(m\) times (for any integer \(m\ge 1\)) gives
\[
(f_\sigma)^{(m)}(z) = \frac{1}{\pi}\sum_{n=m}^{\infty}\frac{\pi^{n}}{(n-m)!}
\sum_{j=0}^{\infty}\frac{\pi^j(j+n-1)^2}{2^{j+n-1}j!}\mathfrak{b}_j \, z^{n-m},\qquad z\in\mathbb{C}.
\tag{2}
\]
In particular, for \(m=1\),
\[
(f_\sigma)'(z) = \frac{1}{\pi}\sum_{n=1}^{\infty}\frac{\pi^{n}}{(n-1)!}
\sum_{j=0}^{\infty}\frac{\pi^j(j+n-1)^2}{2^{j+n-1}j!}\mathfrak{b}_j \, z^{n-1},\qquad z\in\mathbb{C}.
\tag{3}
\]

The theorem states that for every \(m\ge 1\),
\[
(f_\sigma)^{(m)}(z)
=
\sum_{n=m}^{\infty} \frac{n!}{(n-m)!}\, c_n \, S_{n-m}(z) + \lambda_m(f),
\]
where
\[
\lambda_m(f) := \sum_{n=m-1}^{\infty} \frac{n!}{(n-m+1)!}\, \mathfrak b_{n-m+1}\, c_n.
\]
Comparing this with the direct differentiation (2), we see that the constant term of (2) — i.e., the coefficient of \(z^0\) — is exactly \(\lambda_m(f)\). For this specific function,
\[
\lambda_m(f)
=
\pi^{m-1}
\sum_{j=0}^{\infty}
\frac{\pi^j (j+m-1)^2}{2^{j+m-1} j!}\,\mathfrak b_j
\qquad (m\ge 1).
\tag{4}
\]
Indeed, setting \(n=m\) in (2) yields exactly (4). For \(m=1\), (4) reduces to
\[
\lambda_1(f)=\sum_{j=0}^{\infty}\frac{\pi^j j^2}{2^j j!}\mathfrak b_j
=
\sum_{j=0}^{\infty} c_j\mathfrak b_j
=
\lambda(f),
\]
which matches the original \(\lambda(f)\) from (2.1).

Using the identity \(B_r(z+1)=rS_{r-1}(z)+\mathfrak b_r\), the corrected theorem gives the unified form
\[
(f_\sigma)^{(m)}(z)
=
\sum_{n=m-1}^{\infty}
\frac{n!}{(n-m+1)!}\, c_n \, B_{n-m+1}(z+1),
\qquad m\ge 1.
\tag{5}
\]
For this example, substituting \(c_n=\frac{\pi^n n^2}{n!2^n}\) into (5) yields an equivalent expression that automatically includes the constant term \(\lambda_m(f)\) via \(B_0(z+1)=1\).

Termwise integration of (1) gives
\[
\int f_\sigma(z)\,dz
=
\frac{1}{\pi}\sum_{n=1}^{\infty}\frac{\pi^{n}}{(n+1)!}
\sum_{j=0}^{\infty}\frac{\pi^j(j+n-1)^2}{2^{j+n-1}j!}\mathfrak{b}_j \, z^{n+1} + C,\qquad z\in\mathbb{C},
\tag{6}
\]
where \(C\) is an arbitrary constant. One can check also that this matches the  formula \(\int f_\sigma = \sum \frac{c_n}{n+1} S_{n+1} + \mu(f) z + C\).


\end{ex}

\begin{ex}
Let $p\in\mathbb{R}$ be fixed. Define
\[
c_0=0,\qquad
c_n=\frac{\pi^n n^p}{(e^{2\pi n}-1)n!}\quad (n\ge1),
\]
and put
\[
f(z)=\sum_{n=0}^{\infty}c_nz^n
=\sum_{n=1}^{\infty}
\frac{\pi^n n^p}{(e^{2\pi n}-1)n!}z^n,
\qquad z\in\mathbb C.
\]

Since
\[
\sum_{n=0}^{\infty}\frac{n!}{\pi^n}|c_n|
=
\sum_{n=1}^{\infty}
\frac{n^p}{e^{2\pi n}-1}
<\infty,
\]
the function $f$ is absolutely entirely summable by Theorem~\ref{review thm}.

Hence
\[
f_\sigma(z)
=
\frac1\pi
\sum_{n=1}^{\infty}
\frac{\pi^n}{n!}
\sum_{j=0}^{\infty}
\frac{\pi^j(j+n-1)^p}
{(e^{2\pi(j+n-1)}-1)j!}
\mathfrak b_j\,z^n,
\qquad z\in\mathbb C,
\]
where, for $n=1$, the inner sum is interpreted as
\[
\sum_{j=1}^{\infty}
\frac{\pi^j j^p}
{(e^{2\pi j}-1)j!}\mathfrak b_j,
\]
since the contribution corresponding to $j=0$ is precisely the
$c_0$-term, and $c_0=0$.

Since the above power series converges absolutely, it may be
differentiated termwise. Thus
\[
(f_\sigma)'(z)
=
\frac1\pi
\sum_{n=1}^{\infty}
\frac{\pi^n}{(n-1)!}
\sum_{j=0}^{\infty}
\frac{\pi^j(j+n-1)^p}
{(e^{2\pi(j+n-1)}-1)j!}
\mathfrak b_j\,z^{\,n-1},
\qquad z\in\mathbb C,
\]
where again, when $n=1$, the inner sum starts from $j=1$.

More generally, for every integer $m\ge1$,
\[
(f_\sigma)^{(m)}(z)
=
\frac1\pi
\sum_{n=m}^{\infty}
\frac{\pi^n}{(n-m)!}
\sum_{j=0}^{\infty}
\frac{\pi^j(j+n-1)^p}
{(e^{2\pi(j+n-1)}-1)j!}
\mathfrak b_j\,z^{\,n-m},
\qquad z\in\mathbb C.
\]


\end{ex}

\section{Bernoulli polynomials series and the derivative of analytic summand functions}
Some interesting properties and various applications of the topic of analytic summability was presented in \cite{analytic, soodeh}. Now we want to show an interesting fact about the derivative of analytic summand functions which characterized the nature of Bernoulli polynomials sereis of the form $\sum_{n=1}^{\infty}c_nB_n(z)$. Indeed, it will be proved that Bernoulli polynomials sereis is the same derivative of $f_\sigma$ up to a unit forward shift for some entire functions $f$.
 In this way, some identities and criteria for convergence of the Bernoulli polynomials sereis is obtained. Finding another criterion for analytic summability is the other result of this section.\\
Moreover, we study two other functional sequences $f_{\sigma^\prime_N}(z)$, $f^\prime_{\sigma_N}(z)$ and their relation to $(f_\sigma)^\prime_N(z)$. Consider the Bernoulli polynomials series $\sum_{n=0}^{\infty}c_nB_n(z)$ and put $f(z)=\sum_{n=0}^{\infty}c_nz^n$ where $c_n$ is a complex (or real) sequence. We put
\begin{align}\label{fsigmanprime}
f_{\sigma^\prime_N}(z) &:= (f_{\sigma_N}(z))^\prime = \sum_{n=0}^{N}c_nS^\prime_n(z),\;\;f_{\sigma^\prime}(z)=\lim_{n\to\infty}f_{\sigma^\prime_N}(z).
\end{align}
and
\begin{align}\label{f prime sigma01}
f^\prime_{\sigma}(z) &:= (f^\prime)_{\sigma}(z).
\end{align}
Now, a question arised is, what is the relation between the above functions and  the derivative of analytic summand functions ( $(f_\sigma)^\prime(z)$). The following theorem gives an interesting answer to the question.
\begin{theorem}\label{equivalent thm01}
Suppose $c_n$ is a complex sequence and $f(z)=\sum_{n=0}^{\infty}c_nz^n$ is defined on an open domain $D$. Then the following statements are equivalent\\
(a) The function $f_{\sigma^\prime}$ exists on $D$.\\
(b)  The Bernoulli polynomial series $\sum_{n=1}^{\infty}c_nB_n(z+1)$ is convergent on $D$. \\
(c) The series $\sum_{n=1}^{\infty}nc_nS_{n-1}(z)$ is  convergent on $D$.\\
(d) The function $f^\prime$ is  analytically summable on $D$.\\
If one of the above equivalent conditions holds, then we have
\begin{align}\label{identity01}
\sum_{n=1}^{\infty}c_nB_n(z+1)  = f_{\sigma^\prime}(z) = f^\prime_{\sigma}(z) + f_{\sigma^\prime}(0) =\sum_{n=1}^{\infty}nc_nS_{n-1}(z) + f_{\sigma^\prime}(0),\;z\in D,
\end{align}
\textup{(}thus by putting $z=0$ in (\ref{identity01}), we conclude that $\lambda(f)$ in equation (\ref{derivative formula}) is the same $f_{\sigma^\prime}(0)$\textup{)}.
\end{theorem}
\begin{proof}
(a) $\Rightarrow$ (b) According to (\ref{sigmaoprator}) and identity $B^\prime_n(z)=nB_{n-1}(z)$ we have
\begin{align*}
S^\prime_n(z)={\Big(\frac{B_{n+1}(z+1)-B_{n+1}(1)}{{n+1}}\Big)}^\prime =B_n(z+1),
\end{align*}
Therfore
\begin{align}\label{fsigma prime 02}
f_{\sigma^\prime_N}(z) &= \sum_{n=0}^{N}c_nS^\prime_n(z)
= \sum_{n=0}^{N}c_nB_n(z+1).
\end{align}
 Letting $N\to\infty$ gets the result.\\
(b) $\Rightarrow$ (c) Calculating the partial summation below
\begin{align*}
\sum_{n=1}^{N}nc_nS_{n-1}(z) = \sum_{n=0}^{N}c_n\big(B_n(z+1) - B_n(1)\big).
\end{align*}
and letting $N\to\infty$ gets the result (the assumption implies that the series $\sum_{n=0}^{\infty}c_nB_n(1)$ is convergent on $D$).\\
(c) $\Rightarrow$ (d) For the function $f^\prime(z)$ we have
\begin{align}\label{f prime sigma 02}
f^\prime_{\sigma_N}(z) &= (f^\prime)_{\sigma_N}(z) \notag\\&
= \sum_{n=0}^{N}(n+1)c_{n+1}S_n(z)
=  \sum_{n=1}^{N+1}nc_nS_{n-1}(z).
\end{align}
 Letting $N\to\infty$ gets the result.\\
\end{proof}

\begin{cor}\label{equivalent thm02}
Let $c_n$ be a complex sequence and put $f(z)=\sum_{n=0}^{\infty}c_nz^n$. If the series $\sum_{n=0}^{\infty}\frac{n!}{
\pi^n}c_n$ is absolutely convergent, then\\
(a) the series $\sum_{n=1}^{\infty}c_nB_n(z+1)$ and $\sum_{n=1}^{\infty}nc_nS_{n-1}(z)$ are absolutely convergent on $\mathbb{C}$.\\
(b) The functions $(f^{(m)})_{\sigma}(z)$, for $m=0,1,2,...$, and $f_{\sigma^\prime}(z)$ exist on $\mathbb{C}$.\\
Moreover, the following identities hold
\begin{align}\label{identity02}
\sum_{n=1}^{\infty}c_nB_n(z+1) &= (f_\sigma)^\prime(z) = f_{\sigma^\prime}(z) = f^\prime_{\sigma}(z) + f_{\sigma^\prime}(0)\notag\\& =\sum_{n=1}^{\infty}nc_nS_{n-1}(z) +f_{\sigma^\prime}(0),\;z\in \mathbb{C}.
\end{align}
\end{cor}
\begin{proof}
The assertions are proved by applying Theorems \ref{review thm}, \ref{thm derivative} and \ref{equivalent thm01}. Also, Theorems \ref {thm derivative} and \ref{equivalent thm01} imply the identities.
\end{proof}
Since the analytic summand function satisfies (\ref{function equation}), if the conditions of Theorem \ref{equivalent thm01} hold then a similar difference functional equation exists for $f^\prime$ as follows.
\begin{lem}
Let $f(z)=\sum_{n=0}^{\infty}c_nz^n$  be an analytic function. If $f'$ is analytically summable on $D$, then $\phi= f_{\sigma^\prime}- f_{\sigma^\prime}(0)=f^\prime_{\sigma}$ satisfies the following difference functional equation
\begin{align}\label{function equation02}
\phi(z)=f^\prime(z)+\phi(z-1), \;\;z\in D\cap D+1.
\end{align}
Moreover, if the series $\sum_{n=0}^{\infty}\frac{n!}{
\pi^n}c_n$ is absolutely convergent, then
\begin{align*}
f^\prime_{\sigma}(z)=(f_\sigma)^\prime (z) - f_{\sigma^\prime}(0)= f_{\sigma^\prime}(z) - f_{\sigma^\prime}(0),\;z\in\mathbb{C},
\end{align*}
hence, the function $(f_\sigma)^\prime (z) - f_{\sigma^\prime}(0)$ satisfies (\ref{function equation02}) on $\mathbb{C}$.
\end{lem}
Here we prove another test for analytic summability by applying the relation between Bernoulli polynomials series and $\sum_{n=1}^{\infty}nc_nS_{n-1}(z)$ as the following lemma.
\begin{lem}\label{new test 01}
Let $c_n\neq0$ be a complex sequence such that $\big|\frac{c_n}{(n+1)c_{n+1}}\big|$ be convergent to the positive real number $L$ and put $f(z)=\sum_{n=0}^{\infty}c_nz^n$. If the Bernoulli polynomial series $\sum_{n=1}^{\infty}c_nB_n(z+1)$ is absolutely convergent on $D$, where $D$ is the open domain of $f$, then \textup{(}all assertions of Theorm \ref{equivalent thm01} and identity (\ref{identity01}) hold and\textup{)} the function $f$ is absolutely analytically summable on $D$.
\end{lem}
\begin{proof}
According to Theorem \ref{equivalent thm01} the series
\begin{align*}
\sum_{n=1}^{\infty}nc_nS_{n-1}(z)=\sum_{n=0}^{\infty}{(n+1)}c_{n+1}S_{n}(z)
\end{align*}
is absolutely convergent on $D$. Now applying the limit comparison test between two series $\sum_{n=0}^{\infty}{(n+1)}c_{n+1}S_{n}(z)$ and $\sum_{n=0}^{\infty}c_{n}S_{n}(z)$ along with the assumption $\lim_{n\to\infty}\big|\frac{c_n}{(n+1)c_{n+1}}\big|=L>0$ imply that the series $\sum_{n=0}^{\infty}c_{n}S_{n}(z)$ is absolutely convergent on $D$. Hence, $f(z)$ is absolutely analytically summable on $D$.
\end{proof}
The condition of absolutely convergence of the series $\sum_{n=0}^{\infty}\frac{n!}{
\pi^n}c_n$, is the important criteria to check the analytic summability of a function $f(z)=\sum_{n=0}^{\infty}c_nz^n$ (see Theorem \ref{review thm}). Now we get an interesting example such that the condition does not hold but the function $f$ is analytically summable.
\begin{ex}
Consider $c_n=\frac{{(\pi+t_0)}^n}{n!}$, where $|t_0 |<2\pi$ and put $f(z)=\sum_{n=0}^{\infty}c_nz^n=\sum_{n=0}^{\infty}\frac{{(\pi+t_0)}^n}{n!}z^n$. By using (\ref{bnzformula}) we have
\begin{align}
\sum_{n=1}^{\infty}\frac{{(\pi+t_0)}^n}{n!}B_n(z+1)=\frac{(\pi+t_0) e^{(\pi+t_0)(z+1)}}{e^{\pi+t_0}-1},\;\;z\in\mathbb{C},
\end{align}
 Hence, by applying Lemma \ref{new test 01}, $f$ is absolutely analytically summable on $\mathbb{C}$ (because $\lim_{n\to\infty}\big|\frac{c_n}{(n+1)c_{n+1}}\big|=\frac{1}{\pi+t_0}$) and using (\ref{fsigma def}) implies
\begin{align*}
f_\sigma(z)=\sum_{n=1}^{\infty}\frac{{(\pi+t_0)}^n}{(n+1)!}\big(B_{n+1}(z+1)-B_{n+1}(1)\big).
\end{align*}
Moreover, using Theorem \ref{equivalent thm01} we have
\begin{align*}
(f_\sigma)^\prime(z)&=f_{\sigma^\prime}(z)
=\frac{(\pi+t_0) }{e^{\pi+t_0}-1}e^{(\pi+t_0)(z+1)},\;\;z\in \mathbb{C},
\end{align*}
and
\begin{align*}
f^\prime_\sigma(z) =\frac{(\pi+t_0) e^{(\pi+t_0)}}{e^{\pi+t_0}-1}(e^{z+1}-1),\;\;z\in \mathbb{C}.
\end{align*}
\end{ex}
\textbf{Question I.} Is the mentioned open domain $D$ in lemma \ref{new test 01}, the same $\mathbb{C}$?\\
The following lemma represents another test for the analytic summability which requires the condition of uniformly convergence of the Bernoulli polynomial series.
\begin{lem}\label{new test 02}
Assume $c_n$ is a complex (or real) sequence and $f(z)=\sum_{n=0}^{\infty}c_nz^n$ is a function defined on an open domain $D$. If the Bernoulli polynomial series $\sum_{n=1}^{\infty}c_nB_n(z+1)$ is absolutely and uniformly convergent on $D$, then \textup{(}all assertions of Theorm \ref{equivalent thm01} and identity (\ref{identity02}) hold on $D$ and\textup{)} $f$ is analytically summable on $D$.
\end{lem}
\begin{proof}
By applying Theorem \ref{equivalent thm01} and relation (\ref{fsigmaprimeN}) the functional series
 \begin{align*}
 \sum_{n=1}^{\infty}nc_nS_{n-1}(z)+\sum_{n=0}^{\infty}\mathfrak{b}_jc_j&=\sum_{n=1}^{\infty}\frac{1}{(n-1)!}\sum_{j=0}^{\infty}\frac{(j+n-1)!}{j!}\mathfrak{b}_jc_{j+n-1}z^{n-1},
 \end{align*}
 is uniformly convergent to the function $(f_\sigma)^\prime(z)$ on $D$. Considering $f_\sigma(0)=0$ and using Theorem 5.2 of \cite{complex book} imply that the functional sequence $f_{\sigma_N}(z)$ is uniformly convergent. Hence, $f$ is uniformly summable on $D$.  The rest of proof followed by Theorem \ref{equivalent thm01}.
\end{proof}
Now, two questions arise from the above lemma as follows\\
\textbf{Question II.} Is one of the (absolutely or uniformly convergence) conditions of the Bernoulli polynomial series $\sum_{n=1}^{\infty}c_nB_n(z+1)$ extra? \\
\textbf{Question III.} Does the absolutely convergence condition of the series $\sum_{n=0}^{\infty}\frac{n!}{
\pi^n}c_n$ imply the uniformly convergence of the Bernoulli polynomial series?
\begin{ex}
Consider $f(z) = \alpha a^z+\beta \sin(z)$, where $\alpha, \beta \in\mathbb{C}$ and $|\ln a|<\pi$. Theorem 2.1 and examples of section 5 of \cite{analytic} imply that the function $f(z)$ is absolutely entirely summable and we have
\begin{align*}
f_\sigma(z)=\frac{\alpha a}{a-1}(a^z-1)+\beta \frac{\sin(z)+\sin(1)-\sin(z+1)}{2-2\cos(1)},\;z\in\mathbb{C}.
\end{align*}
As it expected from Theorem \ref{thm derivative}, $f_\sigma(z)$ is differntiable on $\mathbb{C}$ and we get
\begin{align*}
(f_\sigma)^\prime(z) =\frac{\alpha\ln a}{a-1}a^{z+1}+\beta \frac{\cos(z)-\cos(z+1)}{2-2\cos(1)},\;z\in\mathbb{C}.
\end{align*}
On the other hand, $f^\prime(z)=\alpha\ln a a^z+\beta\cos(z)$ is absolutely entirely summable (see Theorem \ref{review thm}) and
\begin{align*}
(f^\prime)_\sigma(z)=\frac{\alpha a\ln a}{a-1}(a^z-1)+ \beta\frac{\cos(z)+\cos(1)-\cos(z+1)-1}{2-2\cos(1)}
\end{align*}

Moreover, Theorem \ref{equivalent thm01} implies that the following Bernoulli polynomials series is absolutely convergent on $\mathbb{C}$ and we have
\begin{align*}
\sum_{n=0}^{\infty}\frac{\alpha (\ln a)^n+\beta (-1)^{[\frac{n}{2}]}\epsilon_n}{n!}B_n(z+1) =\frac{\alpha \ln a}{a-1}a^{z+1}+\beta\frac{\cos(z)-\cos(z+1)}{2-2\cos(1)},\;z\in\mathbb{C},
\end{align*}
where $\{\epsilon_n\}$ is a sequence such that $\epsilon_n=0$ if $n$ is even and $\epsilon_n=1$ if $n$ is odd. In particular
\begin{align*}
f_{\sigma^\prime}(0)=\sum_{n=0}^{\infty}\frac{\alpha (\ln a)^n+\beta (-1)^{[\frac{n}{2}]}\epsilon_n}{n!}B_n(1)=\frac{\alpha a\ln a}{a-1}+\frac{\beta}{2}.
\end{align*}
\end{ex}

\section{upper bounds for the Bernoulli polynomials series and derivative of analytic summand functions}
In this section, we prove some upper bounds for the Bernoulli polynomials and Bernoulli polynomials series and as its result, for the derivative of analytic summand functions.  First we make upper bounds for Benoulli polynomials by applying upper bounds of Bernoulli numbers as a lemma below.

\begin{lem}\label{lembound01}
The Bernoulli polynomials satisfy the following inequalities
\begin{align}\label{bnzbound01}
|B_n(z)|&\leq
 \begin{cases}|z|^n+\frac{n|z|^{n-1}}{2}+\sum_{m=1}^{\frac{n-1}{2}}\lambda_j\;\frac{n!}{(2m-1)!\pi^{n-2m+1}}|z|^{2m-1}  & ; n \quad\text{is odd}\\|z|^n+\frac{n|z|^{n-1}}{2}+\sum_{m=1}^{\frac{n}{2}-1}\lambda_j\;\frac{n!}{(2m)!\pi^{n-2m}}|z|^{2m}  & ; n \quad\text{is even}
\end{cases}\notag\\&
\leq \begin{cases}|z|^n+\frac{n|z|^{n-1}}{2}+\sum_{m=1}^{\frac{n-1}{2}}\frac{\pi^2}{12}\;\frac{n!}{(2m-1)!\pi^{n-2m+1}}|z|^{2m-1}  & ; n \quad\text{is odd}\\|z|^n+\frac{n|z|^{n-1}}{2}+\sum_{m=1}^{\frac{n}{2}-1}\frac{\pi^2}{12}\;\frac{n!}{(2m)!\pi^{n-2m}}|z|^{2m}  & ; n \quad\text{is even}
\end{cases}\notag\\&
\leq  |z|^n+\frac{n|z|^{n-1}}{2}+\frac{n!}{\pi^{n}}\sum_{k=0}^{n-2}\frac{(\pi|z|)^k}{k!}
\notag\\&\leq \frac{2n!}{\pi^{n}}\sum_{k=0}^{n}\frac{(\pi|z|)^k}{k!},
\end{align}
for every fixed positive integer $j$, $\lambda_j = \frac{2(2^{2j}-1)}{3\times2^{2j}} \zeta (2j)$ and $n\geq j$.
\end{lem}
\begin{proof}
Inequality (\ref{bnbound3}) together $B_{2r+1}=0$ imply that
\begin{align}
|\binom{n}{k}B_{n-k}|=\frac{n!}{k!(n-k)!}|B_{n-k}|&\leq\frac{n!}{k!(n-k)!}\frac{2(n-k)!}{(\pi)^{n-k}}\frac{2(2^{2j}-1)}{2^{2j}}\zeta(2j)\frac{1}{2^{n-k}-1}\notag\\&
=\frac{n!}{k!\pi^{n-k}}\frac{2(2^{2j}-1)}{2^{2j}}\zeta(2j)\frac{1}{2^{n-k}-1}
\notag\\&\leq\frac{n!}{3k!\pi^{n-k}}\frac{2(2^{2j}-1)}{2^{2j}}\zeta(2j);\quad 0\leq k\leq n-2.
\end{align}
so we have
\begin{align}
|\binom{n}{k}B_{n-k}|\leq \begin{cases}1 & ; k=n\\ \frac{n}{2} & ; k=n-1\\ \lambda_j\frac{n!}{k!\pi^{n-k}} & ;0\leq k\leq n-2.
\end{cases}
\end{align}
Hence
\begin{align*}
|B_n(z)|=|\sum_{k=0}^n\binom{n}{k}B_{n-k}z^k|&\leq
|z|^n+\frac{n|z|^{n-1}}{2}+\sum_{\substack {k=0\\n-k\; \text{is even}}}^{n-2}\lambda_j\frac{n!}{k!\;\pi^{n-k}}|z|^k,
\end{align*}
So, by considering the right partial sum in two cases (where $n$ is odd or even) we arrive at the first inequality (the other inequalities are obtaind by applying (\ref{bnbound1}) in similar way).
\end{proof}
Now we will state and prove some upper bounds for Bernoulli polynomials series also for the derivative of analytic summand functions.
\begin{theorem}\label{bn series}
Let $c_n\neq0$ be a complex (or real) sequence and put $f(z)=\sum_{n=0}^{\infty}c_nz^n$. If the series $\sum_{n=0}^{\infty}\frac{n!}{\pi^n}c_n$ is absolutely convergent, then  we have
\begin{align}\label{series bn ine1}
|\sum_{n=0}^{\infty}c_nB_n(z)|\leq 2 e^{\pi|z|}Abs_{\frac{!}{\pi}}(f),\;z\in\mathbb{C}.
\end{align}
\end{theorem}
\begin{proof}
 For every $N\in\mathbb{N}$ and $t\geq0,$ put
\begin{align*}
\phi_N(t)& :=|c_0| + |c_1z| +|\frac{1}{2}c_1| + \sum_{n=2}^{N}|c_n||z|^n +  \frac{1}{2}\sum_{n=2}^{N}n|c_n||z|^{n-1} \\&
+  \sum_{\substack {n=2\\n\; \text{is odd}}}^{N}\sum_{m=1}^{\frac{n-1}{2}}t|c_n|\;\frac{n!}{(2m-1)!\pi^{n-2m+1}}|z|^{2m-1} \\&+ \sum_{\substack {n=2\\n\; \text{is even}}}^{N}\sum_{m=1}^{\frac{n}{2}-1}t|c_n|\;\frac{n!}{(2m)!\pi^{n-2m}}|z|^{2m}\\&
=|c_0| + |c_1z| +|\frac{1}{2}c_1| + \sum_{n=2}^{N}|c_n||z|^n +  \frac{1}{2}\sum_{n=2}^{N}n|c_n||z|^{n-1}\\&
+ t \sum_{\substack {n=2\\n\; \text{is odd}}}^{N}\frac{n!}{\pi^n}|c_n|\sum_{m=1}^{\frac{n-1}{2}}\frac{(\pi|z|)^{2m-1}}{(2m-1)!} + t \sum_{\substack {n=2\\n\; \text{is even}}}^{N}\frac{n!}{\pi^n}|c_n|\sum_{m=1}^{\frac{n}{2}-1}\frac{(\pi|z|)^{2m}}{(2m)!}.
\end{align*}
A calculus with do attention to Lemma \ref{lembound01}, indicates that
\begin{align*}
|\sum_{n=0}^Nc_nB_n(z)|&\leq |c_0 B_0(z) + c_1B_1(z)| + |\sum_{n=2}^{N}c_nB_n(z)|\\&
\leq |c_0| + |c_1z| +|\frac{1}{2}c_1| + \sum_{n=2}^{N}|c_n| \Big( |z|^n+\frac{n|z|^{n-1}}{2}+\sum_{k=0}^{{n-2}}\binom{n}{k}|B_{n-k}||z|^k\Big)\\&
\leq \phi_N(\lambda_j)
\end{align*}
According to $\frac{2}{3} = \displaystyle\lim_{j\to\infty}\lambda_j \leq \lambda_j \leq \lambda_1= \frac{\pi^2}{12}$ along with $\phi_N(t)$ is an increasing function respect to $t$ on $[0,+\infty)$, we have
$$\phi_N(\frac{2}{3}) \leq \phi_N(\lambda_j) \leq \phi_N(\lambda_1) = \phi_N(\frac{\pi^2}{12}).$$
Now by letting $N\to\infty$ we arrive at the first three inequalities. (the rest of proof similar to theorem 4.1 of \cite{analytic}).
\end{proof}

\begin{cor}\label{bnz bound}
Let $c_n\neq0$ be a complex (or real) sequence and put $f(z)=\sum_{n=0}^{\infty}c_nz^n$. If the series $\sum_{n=0}^{\infty}\frac{n!}{\pi^n}c_n$ is absolutely convergent, then we have
\begin{align*}
|\sum_{n=0}^{\infty}c_nB_n(z)|&\leq Abs(f(z)) + \frac{1}{2}Abs(f^\prime(z)) +\frac{2}{3} \cosh (\pi|z|-1)Abs^e_{\frac{!}{\pi}}(f) \\&+\frac{2}{3} \sinh(\pi|z|)Abs^o_{\frac{!}{\pi}}(f)
- \frac{2}{3} |c_0| \cosh (\pi|z|-1) - \frac{2}{3}\frac{|c_1|}{\pi} \sinh(\pi|z|) \\&
\leq Abs(f(z)) + \frac{1}{2}Abs(f^\prime(z)) +\lambda_j \cosh (\pi|z|-1)Abs^e_{\frac{!}{\pi}}(f) \\&+\lambda_j \sinh(\pi|z|)Abs^o_{\frac{!}{\pi}}(f)
- \lambda_j |c_0| \cosh (\pi|z|-1) - \lambda_j\frac{|c_1|}{\pi} \sinh(\pi|z|) \\&
\leq Abs(f(z)) + \frac{1}{2}Abs(f^\prime(z)) +\frac{\pi^2}{12} \cosh (\pi|z|-1)Abs^e_{\frac{!}{\pi}}(f) \\&+\frac{\pi^2}{12} \sinh(\pi|z|)Abs^o_{\frac{!}{\pi}}(f)
- \frac{\pi^2}{12} |c_0| \cosh (\pi|z|-1) - \frac{\pi^2}{12}\frac{|c_1|}{\pi} \sinh(\pi|z|) \\&
\leq  Abs(f(z)) +  \frac{1}{2}Abs(f^\prime(z)) + \frac{\pi^2}{12} \sinh(\pi|z|)Abs_{\frac{!}{\pi}}(f) \\&- \frac{\pi^2}{12} \sinh(\pi|z|)(\frac{\pi|c_0|+|c_1|}{\pi})\\&
\leq Abs(f(z)) +  \frac{1}{2}Abs(f^\prime(z)) +  e^{(\pi|z|)}Abs_{\frac{!}{\pi}}(f) -  e^{(\pi|z|)}(\frac{\pi|c_0|+|c_1|}{\pi}) \\&
\leq 2 e^{(\pi|z|)}Abs_{\frac{!}{\pi}}(f),\;\; z\in\mathbb{C},\;j=1,2,3, ...\;.
\end{align*}
\end{cor}
\begin{ex}\label{examle2}
Put $c_n = \frac{\pi^n}{n!n^p}$ for $n\geq1$ and $c_0 = 0$. The Bernoulli polynomials series $ \sum_{n=0}^{\infty}c_nB_n(z)$ is absoloutely entirely convergent for every real number $p>1$ (see Corollary \ref{equivalent thm01}) and by applying Corollary \ref{bnz bound} we have
\begin{align*}
|\sum_{n=1}^{\infty}\frac{\pi^n}{n!n^p}B_n(z)| &\leq \sum_{n=1}^{\infty}\frac{\pi^n}{n!n^p}|z|^n +\frac{1}{2}\sum_{n=2}^{\infty}\frac{\pi^n}{(n-1)!n^p}|z|^{n-1} \\&
+\frac{2}{3}(\cosh(\pi|z|)-1)\sum_{k=1}^{\infty}\frac{1}{(2k)^p} + \frac{2}{3}\sin(\pi|z|)  \sum_{k=1}^{\infty}\frac{1}{(2k-1)^p} \\&
= \sum_{n=1}^{\infty}\frac{\pi^n}{n!n^p}|z|^n +\frac{1}{2}\sum_{n=2}^{\infty}\frac{\pi^n}{(n-1)!n^p}|z|^{n-1} \\&
+ \frac{2^{1-p}\zeta(p)}{3}(\cosh(\pi|z|)-1) + \frac{2-2^{1-p}\zeta(p)}{3}\sinh(\pi|z|)\\&
\leq 2e^{\pi|z|}\zeta(p),\;z\in \mathbb{C}.
\end{align*}
As a result we arrive at the following (infinitely many numbers) lower bounds for $\zeta(p)$
\begin{align}\label{zeta bound1}
\zeta(p) \geq \frac{|\sum_{n=1}^{\infty}\frac{\pi^n}{n!n^p}B_n(z)|}{2e^{\pi|z|}},\;\;p>1,\;z\in\mathbb{C}.
\end{align}
\end{ex}
More study on Example \ref{examle2} and inequality (\ref{zeta bound1}), led us to obtain more useful relations for cases $z=0$ or $z=1$ as the lemma below (becuase they have closed-forms).
\begin{lem}
For every real number $p>1$ we have
\end{lem}
\begin{align}\label{tag 3}
\sum_{n=1}^{\infty}\frac{\pi^n}{n!n^p}B_n(0)=\sum^{\infty}_{n=1}\frac{B_{n}\pi^n}{n! n^p}=-\frac{\pi}{2}-\sum_{n\in\mathbb{N}}\textrm{Li}_p\left(\frac{i}{2n}\right)\textrm{, }
\end{align}
\begin{align}\label{tag 33}
\sum_{n=1}^{\infty}\frac{\pi^n}{n!n^p}B_n(1)=\sum^{\infty}_{n=1}\frac{\mathfrak{b}_j\pi^n}{n! n^p}=\frac{\pi}{2}-\sum_{n\in\mathbb{N}}\textrm{Li}_p\left(\frac{i}{2n}\right)\textrm{, }
\end{align}
where $\textrm{Li}_n(z)$ is the polylogarithm function (see \cite{apostel2}).\\ Moreover, if $p\geq 2$ is an integer, then $\zeta(p)$ has the following lower bound
\begin{align}\label{zeta bound2}
\zeta(p)\geq \frac{\big|\frac{\pi}{2}+\frac{(-1)^{p-2}}{(p-2)!}\int^{1}_{0}\frac{(\log t)^{p-2}}{t}\log\left(\frac{2\sinh\left(\frac{\pi t}{2}\right)}{\pi t}\right)dt\big|}{2e^\pi}.
\end{align}
\begin{proof}
If we put $f(x)=\sum^{\infty}_{n=1}\frac{x^n}{n^{p}}=\textrm{Li}_{p}(x)\textrm{, }p>1$, then we have (see chapter 2, pg. 36-39 of \cite{Nikos})
\begin{align}\label{tag 1}
\sum^{\infty}_{n=1}\left(f\left(\frac{x}{2\pi i n}\right)+f\left(\frac{-x}{2\pi i n}\right)-2f(0)\right)=-\sum^{\infty}_{n=1}\frac{f^{(2n)}(0)}{(2n)!}\frac{B_{2n}}{(2n)!}x^{2n}\textrm{, }|x|<2\pi.
\end{align}
It is known that $B_{2n+1}=0$, for $n$ positive integer and $B_1=-1/2$. Hence we can rewrite (\ref{tag 1}) in the form (after setting $x\rightarrow \pi$)
\begin{align}\label{tag 2}
\frac{\pi f'(0)}{2}+\sum^{\infty}_{n=1}\left(f\left(\frac{-i}{2n}\right)+f\left(\frac{i}{2n}\right)-2f(0)\right)=-\sum^{\infty}_{n=1}\frac{B_n}{n!}\pi^n\frac{f^{(n)}(0)}{n!}.
\end{align}
 So we get
\begin{align*}
-\sum^{\infty}_{n=1}\frac{B_{n}\pi^n}{n! n^p}=\frac{\pi}{2}+\sum^{\infty}_{n=1}\left(\textrm{Li}_{p}\left(\frac{-i}{2n}\right)+\textrm{Li}_{p}\left(\frac{i}{2n}\right)\right).
\end{align*}
Hence, we arrive at (\ref{tag 3}) (considering the fact $B_n(1)=B_n(0)$ for all $n>1$ in the above identity, implies (\ref{tag 33})).\\
Relation (\ref{tag 3}) can be simplified to an integral. When $p\geq 2$ (integer), we have (see \cite{apostel2})
\begin{align}\label{tag 4}
\textrm{Li}_p(z)=\frac{(-1)^{p-1}}{(p-2)!}\int^{1}_{0}(\log(t))^{p-2}\frac{\log(1-zt)}{t}dt\textrm{, }|z|\leq 1.
\end{align}
Hence, after simplifications in (\ref{tag 3}) using the identity (this can follow from $2\sinh(\pi t/2)/(\pi t)=\prod^{\infty}_{n=1}\left(1+t^2/(4n^2)\right)$, $t\in\mathbb{C}$ (see \cite{handbook}))
\begin{align}\label{tag 5}
\sum^{\infty}_{n=1}\left(\log\left(1+\frac{it}{2n}\right)+\log\left(1-\frac{it}{2n}\right)\right)=\log\left(\frac{2\sinh\left(\frac{\pi t}{2}\right)}{\pi t}\right),
\end{align}
we get ($p=2,3,\ldots$):
\begin{align}\label{tag 6}
\sum^{\infty}_{n=1}\frac{B_{n}\pi^n}{n! n^p}=-\frac{\pi}{2}+\frac{(-1)^{p-2}}{(p-2)!}\int^{1}_{0}\frac{(\log t)^{p-2}}{t}\log\left(\frac{2\sinh\left(\frac{\pi t}{2}\right)}{\pi t}\right)dt,
\end{align}
\begin{align}\label{tag 66}
\sum^{\infty}_{n=1}\frac{\mathfrak{b}_j\pi^n}{n! n^p}=\frac{\pi}{2}+\frac{(-1)^{p-2}}{(p-2)!}\int^{1}_{0}\frac{(\log t)^{p-2}}{t}\log\left(\frac{2\sinh\left(\frac{\pi t}{2}\right)}{\pi t}\right)dt,
\end{align}
(Relation (\ref{tag 6}) is taken by using (\ref{tag 4}) and (\ref{tag 5}) in (\ref{tag 3})).
 Now applying (\ref{zeta bound1}) implies (\ref{zeta bound2}).
\end{proof}

\subsection*{Acknowledgement}
The authors would like to express their sincere thanks to the individuals who responded to or provided
 valuable feedback on our questions related to this paper on the Math.StackExchange website.

\end{document}